\documentclass[11pt,a4paper]{amsart}

\usepackage{cmap}
\usepackage[T1]{fontenc}

\usepackage{amsmath, amsfonts, amssymb, amsthm}
\usepackage{color}
\usepackage[pdftex]{graphicx}
\usepackage{thmtools}
\usepackage{thm-restate}

\usepackage{tikz}
\usetikzlibrary{matrix,arrows,decorations.pathmorphing}

\usepackage{hyperref}

\usepackage[utf8]{inputenc}
\usepackage[english]{babel}

\usepackage[mathscr,mathcal]{euscript}

\usepackage{csquotes}

\newcommand{\maintheoremref}{the Main Theorem}
\declaretheorem[name=Proposition,style=plain,numberwithin=section]{prop}
\declaretheorem[name=Lemma,style=plain,sibling=prop]{lem}

\declaretheorem[name=Definition,style=definition,sibling=prop]{defn}

\hypersetup{
    colorlinks=false,
    pdfborder={0 0 0},
}

\newcommand{\N}{\mathbb N}
\newcommand{\Z}{\mathbb Z}
\newcommand{\R}{\mathbb R}

\newcommand{\oTM}{\mathring TM}
\newcommand{\oTpM}[1]{\mathring T_{#1}M}

\newcommand{\im}{\operatorname{im}}
\newcommand{\codim}{\operatorname{codim}}
\newcommand{\id}{\operatorname{id}}
\newcommand{\dom}{\operatorname{dom}}
\newcommand{\tensorslot}{\,\cdot\,}

\newcommand{\fibprod}[1]{\times_{\!\scalebox{.6}{\(#1\)}}}

\newcommand{\nongen}[1]{\mathfrak N_{#1}}
\newcommand{\TLV}{(\oTM \fibprod{M} J^1L)^{\neq}}
\newcommand{\LLV}{(\oTM \fibprod{M} J^1L)^=}
\newcommand{\nongenneq}[1]{\nongen{#1}^{\neq}}
\newcommand{\nongenz}[1]{\nongen{#1}^=}

\newcommand{\quotient}[2]{\left. #1 \right/ #2}

\newcommand{\cwedge}{{~\wedge\!\!\!\!\!\bigcirc~}}

\newcommand{\doubleflat}[1]{#1^\flat \otimes #1^\flat}
\newcommand{\algcurv}{\mathcal R}

\title{Generic metrics satisfy the generic condition}
\author{Eric Larsson}
\address{
	Department of Mathematics\\
	KTH Royal Institute of Technology\\
	SE-100 44, Stockholm\\
	Sweden
}
\email[Eric Larsson]{ericlar@kth.se}

\begin{document}

\maketitle

\begin{abstract}
We prove that the \enquote{generic condition} used in singularity theorems of general relativity is generic in the space of Lorentzian metrics on a given manifold, in the sense that it is satisfied for all metrics in a residual set in the Whitney \(C^k\)-topology, for \(k\) depending on the dimension of the manifold.
\end{abstract}

\section{Introduction}
In singularity theorems, for instance the Hawking--Penrose singularity theorem \cite[Theorem 2, Chapter 8, page 266]{HawkingEllis}, a condition called the \emph{generic condition} is imposed on the spacetime.
In index notation, this condition states that for each inextendible timelike or lightlike geodesic \(\gamma\) there is some point \(\gamma(t)\) at which \(\dot\gamma^e\dot\gamma^f\dot\gamma_{[a}R_{b]ef[c}\dot\gamma_{d]} \neq 0\).
The condition can be written in index-free notation using the Kulkarni--Nomizu product, as we do in Definitions~\ref{def:Generic-vector} and \ref{def:Generic-condition-versions}.
For a general discussion of singularity theorems and the role of the generic condition, see \cite{Senovilla98}.

It has been proposed (see for instance \cite[page 101]{HawkingEllis}) that this \enquote{generic condition} should be \enquote{generic} (in some suitable sense) among Lorentzian metrics on a given manifold.
In a paper \cite{BeemHarris93a} called \enquote{The Generic Condition Is Generic} and a follow-up paper \cite{BeemHarris93b} Beem and Harris proved, among other things, that if there is a sufficiently large set of vectors \(X \in T_pM\) satisfying \(X^e X^f X_{[a} R_{b]ef[c} X_{d]} = 0\) then the curvature tensor at \(p\) has a very restricted form.
This suggests that violations of the generic condition should be rare.
However, this analysis concerns only a single point and does not in itself completely answer the question of whether the generic condition is generic in a more global sense.

We will prove that the set of metrics which satisfy the generic condition form a residual subset, in other words a countable intersection of dense open sets, of the space of all Lorentzian metrics on a fixed manifold, when this space is given the Whitney \(C^k\)-topology, for \(k\) depending on the dimension of the manifold.
The globalization of the argument is done by using a transversality theorem.
A similar method was used by Rendall in \cite{Rendall88a} and \cite{Rendall88b} to prove different genericity statements.

To express the generic condition in index-free notation we use the Kulkarni--Nomizu product (see \cite[Definition 1.110]{Besse}).
\begin{defn}
Let \(h\) and \(k\) be symmetric \(2\)-tensors.
Their \emph{Kulkarni--Nomizu product} \(h \cwedge k\) is defined by
\[\begin{aligned}
&(h \cwedge k)(x, y, u, v) \\
&\qquad= h(x, u)k(y, v) + h(y, v)k(x, u) - h(x, v)k(y, u) - h(y, u)k(x, v).
\end{aligned}\]
\end{defn}

\begin{defn}\label{def:Generic-vector}
Let \((M, g)\) be a Lorentzian manifold with \((0, 4)\) curvature tensor \(R\).
We say that a vector \(X \in TM\) is \emph{generic} if
\[(\doubleflat{X}) \cwedge R(\tensorslot, X, \tensorslot, X) \neq 0.\]
In index notation this condition reads \(X^e X^f X_{[a} R_{b]ef[c} X_{d]} \neq 0\), so our definition is equivalent to the one used in \cite{BeemHarris93a} and \cite{BeemHarris93b}.
\end{defn}

\begin{defn}
Let \((M, g)\) be a Lorentzian manifold with \((0, 4)\) curvature tensor \(R\).
Let \(X\) be a vector and let \(r\) be a nonnegative integer.
Let \(\gamma\) be a geodesic segment with \(\dot \gamma(0) = X\).
We say that the vector \(X\) is \emph{\(r\)-nongeneric} if for all integers \(k \in [0, r]\) it holds that
\[
\nabla^k_{\dot\gamma(0)} \left((\doubleflat{\dot\gamma}) \cwedge R(\tensorslot, \dot\gamma, \tensorslot, \dot\gamma)\right) = 0.
\]
\end{defn}

\begin{defn}\label{def:Generic-condition-versions}
A Lorentzian manifold satisfies the \emph{timelike generic condition} if every inextendible timelike geodesic \(\gamma\) has some point at which \(\dot\gamma\) is generic.
The \emph{lightlike generic condition} and \emph{spacelike generic condition} are defined analogously.

We say that a Lorentzian manifold satisfies the \emph{generic condition} if it satisfies the timelike and lightlike generic conditions.
\end{defn}

We will show that each of these generic conditions is generic in the space of all Lorentzian metrics, in the sense made precise in the following theorem.
\begin{restatable}[Generic metrics satisfy the generic condition]{thm}{mainthm}
Let \(M\) be a smooth manifold of dimension \(n \geq 3\).
Let \(r\) be an integer such that
\[
r > \frac{4n - 2}{(n-1)(n-2)}.
\]
Let \(k \in \Z \cup \{\infty\}\) be such that \(k \geq r + 2\).
Let \(L\) denote the fiber bundle of inner products of Lorentzian signature on \(TM\) and endow the set \(\Gamma^\infty(L)\) of smooth sections of \(L\) (i.e.\ the set of Lorentzian metrics on \(M\)) with the Whitney \(C^k\)-topology.
Then there is a residual and dense set \(G \subseteq \Gamma^\infty(L)\) such that if \(g \in G\) then no nonzero vector is \(r\)-nongeneric in \(g\).

This implies that if \(g \in G\) then the points at which \((\doubleflat{\dot\gamma}) \cwedge R(\tensorslot, \dot\gamma, \tensorslot, \dot\gamma)\) is zero along any \(g\)-geodesic \(\gamma\) form a discrete set.
In particular,
\begin{itemize}
\item the lightlike generic condition holds for each metric \(g \in G\),
\item the timelike generic condition holds for each metric \(g \in G\),
\item the spacelike generic condition holds for each metric \(g \in G\).
\end{itemize}
\end{restatable}
The theorem tells us that for generic metrics, each geodesic has a dense subset where the tangent vectors are generic.
This statement is much stronger than the generic condition, which demands only that this holds at at least one point along the geodesic, not on a dense set.
The methods we will use to prove the theorem are well-suited for obtaining properties on dense sets, but do not capture the concept of \enquote{at least one point}.

It is also worth noting that the theorem tells us that the generic condition is generic in the space of all Lorentzian metrics.
It would perhaps be more interesting to be able to prove that the generic condition is generic in a set of Lorentzian metrics satisfying some additional condition, for instance that of being Ricci flat.
To obtain such a theorem, one might need to adapt the methods to capture the notion of \enquote{at least one point}.
This global property of the generic condition is what makes it difficult to work with, and by proving a much stronger conclusion we sidestep this problem altogether.

The method we will use to prove the theorem is inspired by the method used by Rendall in \cite{Rendall88a} and \cite{Rendall88b}.

\subsection{Notation}
Throughout the paper, \(M\) will denote a smooth manifold of dimension \(n\).
The fiber product of fiber bundles over \(M\) will be denoted by \(\fibprod{M}\).
The \(r\)-fold Whitney sum \(E \oplus E \oplus \cdots \oplus E\) of a vector bundle \(E\) over \(M\) will be denoted by \(E^{\oplus r}\).
We will use \(\oTM\) to denote the tangent bundle without its zero section.
In other words, its fiber \(\oTpM{p}\) over \(p\) is \(T_pM \setminus \{0\}\).
The fiber bundle of inner products of Lorentzian signature on \(TM\) will be denoted by \(L \to M\).
The space of smooth sections of \(L\), in other words the space of smooth Lorentzian metrics on \(M\), will be denoted by \(\Gamma^\infty(L)\).
The \(k\)-jet bundle of \(L\), as described in \autoref{app:jets-and-whitney-topologies}, will be denoted by \(J^kL\), and the \(k\)-jet of a metric \(g\) evaluated at \(p \in M\) will be denoted by \(j^k_pg\).

\section{Surjectivity of the curvature computation map}
In this section we will show that the map which computes a curvature tensor and its derivatives (in a fixed direction) from a metric and its derivatives is a submersion.
This property is used in \autoref{sec:transversality} to determine the codimensions of the inverse images of manifolds constructed in \autoref{sec:nongenericity} to encode nongenericity of vectors.
We begin with a lemma which reduces the submersivity of a fiber bundle map to submersivity of its restrictions to individual fibers.
\begin{lem}\label{lem:submersivity-lemma}
Let \(B\) and \(B'\) be smooth manifolds.
Let \(E \xrightarrow{\pi} B\) and \(E' \xrightarrow{\pi'} B'\) be smooth fiber bundles.
Let \(\Phi \colon E \to E'\) be a smooth bundle map which projects to a smooth map \(\phi \colon B \to B'\).
Suppose that \(\phi\) is a submersion and that \(\Phi\) is a fiberwise submersion (in the sense that for each \(b \in B\) the fiberwise restriction $\Phi_b \colon E_b \to E'_{\phi(b)}$ is a submersion).
Then \(\Phi\) is a submersion.
\end{lem}
\begin{proof}
Let \(F\) and \(F'\) be spaces which are diffeomorphic to the fibers of \(E\) and \(E'\).
Choose some point $e \in E$, and let $b = \pi(e)$.
We will show that \(\Phi\) is submersive at \(e\).
Choose an open trivializing neighborhood \(p \in U \subseteq B\) for \(E\).
Let \(U' = \phi(U)\).
The map \(\phi\) is a submersion, and hence open, so \(U'\) is an open neighborhood of \(\phi(p)\) in \(B'\).
After possibly shrinking \(U\), the set \(U'\) is a trivializing neighborhood for \(E'\).
With respect to these trivializations, the map \(\Phi\) can be written as
\[
\Phi(b, f)
=
(\phi(b), \Phi_b(f)).
\]
Since \(\Phi_b\) is a submersion, the image of the tangent map of this map contains the linear subspace \(\{0\} \times T_{\Phi_b(f)}F' \subseteq T_{\phi(b)}B' \times T_{\Phi_b(f)}F'\).
Since \(\phi\) is a submersion, it also contains a linear subspace of dimension \(\dim B'\) which intersects \(\{0\} \times T_{\Phi_b(f)}F'\) only in \(0\).
Hence the tangent map is surjective, proving that \(\Phi\) is submersive at \(e\).
\end{proof}

We now turn to the definition of maps \(\alpha_r\) which compute curvature tensors and their derivatives.
\begin{defn}
Given a smooth manifold \(M\), let \(\algcurv \to M\) be the vector bundle of tensors with the symmetries of Riemann curvature tensors, in other words the vector bundle of \((0, 4)\)-tensors \(R\) such that
\[R(x, y, v, w) = -R(y, x, v, w),\]
\[R(x, y, v, w) = -R(x, y, w, v),\]
\[R(x, y, v, w) + R(y, v, x, w) + R(v, x, y, w) = 0.\]
\end{defn}

\begin{defn}
For each \(r \in \N\), we will define a map
\[
\alpha_r \colon
	\oTM \fibprod{M} J^{1+r}L
\to
	\oTM \fibprod{M} J^1L \fibprod{M} \algcurv^{\oplus r}.
\]
Here \(\oTM\) denotes the tangent bundle without its zero section.
For \(r = 0\), define \(\alpha_0 \colon \oTM \fibprod{M} J^1L \to \oTM \fibprod{M} J^1L\) to be the identity.
For \(r \geq 1\), proceed as follows.
For \((X, [g]) \in \oTM \fibprod{M} J^{1+r}L\), based at a point \(p \in M\), let \(g\) be a representative of \([g]\).
Explicitly, \(g\) is a Lorentzian metric on a neighborhood of \(p\).
Extend \(X\) to a parallel vector field on the \(g\)-geodesic starting with \(X\).
Let \(R\) be the curvature tensor of \(g\).
For \(r \geq 1\), define
\[
\alpha_r(X, [g])
=
(X, j^1_pg, R, \nabla_X R, \nabla^2_X R, \ldots, \nabla^{r-1}_X R)
\]
(where \(R\) and its covariant derivatives are evaluated at \(p\)).
The maps \(\alpha_r\) are well-defined since the curvature tensor and its derivatives of orders up to \(r-1\) along a geodesic starting with \(X\) can be computed pointwise in terms of \(X\) and the derivatives of the metric of orders up to \(r + 1\).
It is smooth since \(R\) and its derivatives depend smoothly on \(X\) and \(g\).
\end{defn}

\begin{lem}\label{lem:alpha-submersion}
For each \(r \in \N\), the map \(\alpha_r\) is a submersion.
\end{lem}
\begin{proof}
The map \(\alpha_0\) is the identity, which is a submersion.
For \(r \geq 1\) we proceed by induction.
Suppose that \(\alpha_{r-1}\) is a submersion.
Consider the fiber bundles
\[
\oTM \fibprod{M} J^{1+r}L \to \oTM \fibprod{M} J^rL,
\]
\[\oTM \fibprod{M} J^1L \fibprod{M} \algcurv^{\oplus r} \to \oTM \fibprod{M} J^1L \fibprod{M} \algcurv^{\oplus (r-1)}.\]
We can use \(\alpha_r\) and \(\alpha_{r-1}\) to obtain a fiber bundle map
\begin{center}
\begin{tikzpicture}[scale=1.4]
\node (J1rL) at (0,1) {\(\oTM \fibprod{M} J^{1+r}L\)};
\node (JrL) at (0,0) {\(\oTM \fibprod{M} J^rL\)};
\node (Rr) at (4,1) {\(\oTM \fibprod{M} J^1L \fibprod{M} \algcurv^{\oplus r}\)};
\node (Rrm1) at (4,0) {\(\oTM \fibprod{M} J^1L \fibprod{M} \algcurv^{\oplus (r-1)}\).};
\path[->,>=angle 90]
(J1rL) edge node[above]{\(\alpha_r\)} (Rr)
(JrL) edge node[above]{\(\alpha_{r-1}\)} (Rrm1)
(J1rL) edge node[above]{} (JrL)
(Rr) edge node[above]{} (Rrm1);
\end{tikzpicture}
\end{center}
The map \(\alpha_{r-1}\) is a submersion by the induction hypothesis, so if we can show that \(\alpha_r\) is a fiberwise submersion then \autoref{lem:submersivity-lemma} will tell us that it is a submersion.
Fix some \(q \in \oTM \fibprod{M} J^rL\).
We will show that the restriction of \(\alpha_r\) to the fiber \((\oTM \fibprod{M} J^{1+r}L)_q\) over \(q\) is a submersion.

\newcommand{\dindexlist}{i_1 \cdots i_{r-1}}
An element of \((\oTM \fibprod{M} J^{1+r}L)_q\) corresponds to a tensor which in coordinates can be suggestively denoted by \(g_{ab,cd \dindexlist}\), since it is the collection partial derivatives of order \(r+1\) of a metric in coordinates.
More formally, the fiber \((\oTM \fibprod{M} J^{1+r}L)_q\) is isomorphic as a vector space to the space \(S_2 \otimes S_{r+1}\) of \((0, 3+r)\)-tensors which are symmetric in the first two arguments and in the last \(r+1\) arguments.\footnote{Here we have used \(S_k\) to denote the space of completely symmetric \((0, k)\)-tensors on the tangent space \(T_{\pi(q)}M\) where \(\pi \colon \oTM \fibprod{M} J^rL \to M\) is the fiber bundle projection.}
The coordinate expression for the curvature tensor in terms of the metric and Christoffel symbols is
\[
R_{abcd}
=
\frac{1}{2}\left(
	g_{ad,bc} + g_{bc,ad} - g_{ac,bd} - g_{bd,ac}
\right)
+
g_{np} \left(
	\Gamma^n_{bc} \Gamma^p_{ad} - \Gamma^n_{bd} \Gamma^p_{ac}
\right).
\]
This means that the coordinate expression for \(\nabla^{r-1}_XR\) is
\[\begin{aligned}
(\nabla^{r-1}_XR)_{abcd}
&=
X^{i_1}X^{i_2} \cdots X^{i_{r-1}} R_{abcd;\dindexlist}
\\
&\,
\begin{aligned}
= X^{i_1} \cdots X^{i_{r-1}}
\frac{1}{2}\big(
	&g_{ad,bc \dindexlist} + g_{bc,ad \dindexlist}\\
- &g_{ac,bd \dindexlist} - g_{bd,ac \dindexlist}
\big)
\end{aligned}\\
&\qquad + C(q)
\end{aligned}\]
where \(C(q) \in \algcurv\) depends only on \(X\) and the derivatives of \(g\) of order at most \(r\), in other words \(C(q)\) depends only on \(q\).
Without loss of generality, we may work in coordinates where \(X^1 = 1\) and \(X^i = 0\) for \(i \neq 1\).
Then
\renewcommand{\dindexlist}{1 \cdots 1}
\[
(\nabla^{r-1}_XR)_{abcd}
=
\frac{1}{2}\left(
	g_{ad,bc \dindexlist} + g_{bc,ad \dindexlist} - g_{ac,bd \dindexlist} - g_{bd,ac \dindexlist}
\right)
+
C(q).
\]
Hence it holds for \(Q \in (\oTM \fibprod{M} J^{1+r}L)_q\) that
\[
\alpha_r(Q)
=
\frac{1}{2}\left(
	Q_{adbc \dindexlist} + Q_{bcad \dindexlist} - Q_{acbd \dindexlist} - Q_{bdac \dindexlist}
\right)
+
C(q).
\]
This is a linear map from \(S_2 \otimes S_{r+1}\) to \(\algcurv\).
Our goal is to show that it is a submersion, which by linearity is equivalent to it being surjective.
Choose \(P \in \algcurv\), let \(\hat P = P - C(q)\) and let \(Q \in S_2 \otimes S_{r+1} \cong (\oTM \fibprod{M} J^{1+r}L)_q\) be such that
\[
Q_{abcd \dindexlist}
=
-\frac{1}{3}\left(
	\hat P_{acbd} + \hat P_{adbc}
\right)\!.
\]
Then a computation involving the symmetries of \(\hat P \in \algcurv\) shows that
\[
\alpha_r(Q) = \hat P + C(q) = P,
\]
proving that the restriction of \(\alpha_r\) to the fiber \((\oTM \fibprod{M} J^{1+r}L)_q\) is a submersion.
By \autoref{lem:submersivity-lemma} and the induction hypothesis, \(\alpha_r\) is then a submersion, completing the proof.
\end{proof}

\section{Manifolds encoding nongenericity}\label{sec:nongenericity}
We will now construct a set which encodes nongenericity of vectors in different metrics.
Then, we will show that its intersection with the spaces of timelike and lightlike vectors are smooth manifolds, and compute the dimensions of these intersections.
\begin{defn}
For each \(r \in \N\) let \(\nongen{r} \subseteq \oTM \fibprod{M} J^1L \fibprod{M} \algcurv^{\oplus r}\) be the set of elements \((X, [g], R_0, R_1, \ldots R_{r-1})\) such that for each \(i \in \{0, 1, \ldots, r-1\}\) and all \(A, B \in T_pM\) which are \([g]\)-orthogonal to \(X\) it holds that
\[R_i(A, X, B, X) = 0.\]
Here \(T_pM\) is the fiber to which \(X\) belongs.
\end{defn}
The purpose of these sets is shown in the following lemma.
\begin{lem}
Let \((M, g)\) be a Lorentzian manifold.
Then a nonzero vector \(X \in \oTpM{p}\) is \(r\)-nongeneric if and only if \(\alpha_{r+1}(X, j^{r+2}_pg) \in \nongen{r+1}\).
\end{lem}
\begin{proof}
Let \(\gamma \colon (-a, a) \to M\) be a segment of the geodesic with initial velocity \(\dot \gamma(0) = X\).
That \(X\) is \(r\)-nongeneric is by definition the statement that
\[
\nabla^k_{\dot\gamma(0)} \left((\doubleflat{\dot\gamma}) \cwedge R(\tensorslot, \dot\gamma, \tensorslot, \dot\gamma)\right) = 0.
\]
for \(k \in \{0, 1, \ldots, r\}\).
Since \(\gamma\) is a geodesic we have \(\nabla_{\dot\gamma} \dot\gamma = 0\), so it holds that
\[\begin{aligned}
&\nabla^k_{\dot\gamma(0)} \left((\doubleflat{\dot\gamma}) \cwedge R(\tensorslot, \dot\gamma, \tensorslot, \dot\gamma)\right)\\
&\quad= \left(\doubleflat{\dot\gamma(0)}\right) \cwedge (\nabla^k_{\dot\gamma(0)} R)(\tensorslot, \dot\gamma(0), \dot\gamma(0), \tensorslot).
\end{aligned}\]
In other words, \(X\) being \(r\)-nongeneric is equivalent to
\[
\doubleflat{X} \cwedge (\nabla^k_X R)(\tensorslot, X, \tensorslot, X) = 0 \quad \forall k \in \{0, 1, \ldots, r\}.
\]
A computation by Beem and Harris \cite[Proposition 2.2]{BeemHarris93a} tells us that this is equivalent to
\[
(\nabla^k_XR)(A, X, B, X) = 0 \quad \text{ for all } A, B \in T_pM \text{ which are \(g\)-orthogonal to } X.
\]
That this holds for all \(k \in \{0, 1, \ldots, r\}\) is by definition of \(\nongen{r+1}\) equivalent to that \({\alpha_{r+1}(X, j^{r+2}_pg) \in \nongen{r+1}}\), completing the proof.
\end{proof}

We have now defined \(\nongen{r}\) as a set and interpreted its elements in terms of nongenericity.
In the remainder of this section we will show that \(\nongen{r}\) is the union of two smooth manifolds and compute the dimensions of these manifolds.

Let
\[\TLV = \{(X, [g]) \in \oTM \fibprod{M} J^1L \mid [g](X) \neq 0\},\]
\[\LLV = \{(X, [g]) \in \oTM \fibprod{M} J^1L \mid [g](X) = 0\}.\]
These are submanifolds of \(\oTM \fibprod{M} J^1L\) with codimensions \(0\) and \(1\), respectively.
For shorter notation, let
\[\nongenneq{r} = \nongen{r} \cap \left(\TLV \fibprod{M} \algcurv^{\oplus r}\right)\!,\]
\[\nongenz{r} = \nongen{r} \cap \left(\LLV \fibprod{M} \algcurv^{\oplus r}\right)\!.\]

\begin{prop}\label{prop:manifoldness-nonlightlike}
The set \(\nongenneq{r}\) is a smooth submanifold of \(\oTM \fibprod{M} J^1L \fibprod{M} \algcurv^{\oplus r}\) with codimension \(rn(n-1)/2\).
\end{prop}
\begin{proof}
Since \(\TLV\) has codimension \(0\) in \(\oTM \fibprod{M} J^1L\), it is sufficient to prove that \(\nongenneq{r}\) is a submanifold of \(\TLV \fibprod{M} \algcurv^{\oplus r}\) with codimension \(rn(n-1)/2\).

We will characterize the set \(\nongenneq{r}\) as the zero set of a submersion
\[
\TLV \fibprod{M} \algcurv^{\oplus r} \to (S_2E)^{\oplus r},
\]
where \(E\) is a vector bundle over \(\TLV\) which we shall presently construct.
Let \(E \subset TM \fibprod{M} \TLV\) be the subbundle of the vector bundle \(TM \fibprod{M} \TLV \to \TLV\) defined by
\[
E
=
\{(A, X, [g]) \in TM \fibprod{M} \TLV \mid [g](A, X) = 0\}.
\]
In other words, the fiber over \((X, [g])\) consists of vectors \(A\) which are \([g]\)-orthogonal to \(X\).

Now define a vector bundle map
\begin{center}
\begin{tikzpicture}[scale=1.5]
\node (algcurv) at (0,1) {\(\TLV \fibprod{M} \algcurv\)};
\node (TLV) at (1.9,0) {\(\TLV\)};
\node (S2E) at (3,1) {\(S_2E\)};
\path[->,>=angle 90]
(algcurv) edge node[above]{\(c\)} (S2E)
(algcurv) edge node[above]{} (TLV)
(S2E) edge node[above]{} (TLV);
\end{tikzpicture}
\end{center}
by letting
\[
c(X, [g], R)
=
\left(
	(v, w) \mapsto
		R(v, X, w, X)
\right).
\]
This map is a fiberwise surjective linear map, since every symmetric \((0, 2)\)-tensor on a fiber \(E_{(X, [g])}\) is obtained as \(R(\tensorslot, X, \tensorslot, X)\) for some \(R \in \algcurv\).
(Note that \(X \notin E_{(X, [g])}\).
This is necessary for the map to be surjective.)
By combining \(r\) copies of this map, we obtain a vector bundle map
\begin{center}
\begin{tikzpicture}[scale=1.5]
\node (algcurv) at (0,1) {\(\TLV \fibprod{M} \algcurv^{\oplus r}\)};
\node (TLV) at (1.9,0) {\(\TLV\).};
\node (S2E) at (3,1) {\((S_2E)^{\oplus r}\)};
\path[->,>=angle 90]
(algcurv) edge node[above]{\(c \oplus \cdots \oplus c\)} (S2E)
(algcurv) edge node[above]{} (TLV)
(S2E) edge node[above]{} (TLV);
\end{tikzpicture}
\end{center}
The map \(c \oplus \cdots \oplus c\) is a fiberwise submersion, since each \(c\) is a fiberwise surjective linear map.
By \autoref{lem:submersivity-lemma}, this map is then a submersion of total spaces.
Hence the inverse image of the zero section in \((S_2E)^{\oplus r}\) is a submanifold of \(\TLV \fibprod{M} \algcurv^{\oplus r}\) with the same codimension as the codimension of the zero section in \((S_2E)^{\oplus r}\).
This zero section has codimension \(rn(n-1)/2\) since \(E\) has rank \(n-1\).
The set \(\nongenneq{r}\) coincides by definition with this inverse image, completing the proof.
\end{proof}

The corresponding proof for \(\nongenz{r}\) is very similar, but yields a different codimension.
\begin{prop}\label{prop:manifoldness-lightlike}
The set \(\nongenz{r}\) is a smooth submanifold of \(\oTM \fibprod{M} J^1L \fibprod{M} \algcurv^{\oplus r}\) with codimension \(r(n-1)(n-2)/2 + 1\).
\end{prop}
\begin{proof}
Since \(\LLV\) has codimension \(1\) in \(\oTM \fibprod{M} J^1L\), it is sufficient to prove that \(\nongenz{r}\) is a submanifold of \(\LLV \fibprod{M} \algcurv^{\oplus r}\) with codimension \(r(n-1)(n-2)/2\).

We will characterize the set \(\nongenz{r}\) as the zero set of a submersion
\[
\LLV \fibprod{M} \algcurv^{\oplus r} \to (S_2E)^{\oplus r},
\]
where \(E\) is a vector bundle over \(\LLV\) which we shall presently construct.
Let \(E\) be the vector bundle defined by letting the fiber over \((X, [g]) \in \LLV\) be the quotient
\[
E_{(X, [g])}
=
\quotient{\{A \in T_pM \mid [g](A, X) = 0\}}{\R X}.
\]
(Here \(p\) is the image of \((X, [g])\) under the projection \(\LLV \to M\).)
In other words, the fiber over \((X, [g])\) consists of vectors \(A\) which are \([g]\)-orthogonal to \(X\), modulo multiples of \(X\).

Now define a vector bundle map
\begin{center}
\begin{tikzpicture}[scale=1.5]
\node (algcurv) at (0,1) {\(\LLV \fibprod{M} \algcurv\)};
\node (LLV) at (1.9,0) {\(\LLV\)};
\node (S2E) at (3,1) {\(S_2E\)};
\path[->,>=angle 90]
(algcurv) edge node[above]{\(c\)} (S2E)
(algcurv) edge node[above]{} (LLV)
(S2E) edge node[above]{} (LLV);
\end{tikzpicture}
\end{center}
by letting
\[
c(X, [g], R)
=
\left(
	\left([v], [w]\right) \mapsto
		R(v, X, w, X)
\right).
\]
This definition is independent of the choice of representatives of \([v]\) and \([w]\) since \(R(X, X, \tensorslot, X) = R(\tensorslot, X, X, X) = 0\).
The map is a fiberwise surjective linear map, since every symmetric \((0, 2)\)-tensor on a fiber \(E_{(X, [g])}\) is obtained as \(R(\tensorslot, X, \tensorslot, X)\) for some \(R \in \algcurv\).
(Note that the quotient in the definition of \(E\) is necessary for the map to be surjective.)
By combining \(r\) copies of this map, we obtain a vector bundle map
\begin{center}
\begin{tikzpicture}[scale=1.5]
\node (algcurv) at (0,1) {\(\LLV \fibprod{M} \algcurv^{\oplus r}\)};
\node (LLV) at (1.9,0) {\(\LLV\).};
\node (S2E) at (3,1) {\((S_2E)^{\oplus r}\)};
\path[->,>=angle 90]
(algcurv) edge node[above]{\(c \oplus \cdots \oplus c\)} (S2E)
(algcurv) edge node[above]{} (LLV)
(S2E) edge node[above]{} (LLV);
\end{tikzpicture}
\end{center}
The map \(c \oplus \cdots \oplus c\) is a fiberwise submersion, since each \(c\) is a fiberwise surjective linear map.
By \autoref{lem:submersivity-lemma}, this map is then a submersion of total spaces.
Hence the inverse image of the zero section in \((S_2E)^{\oplus r}\) is a submanifold of \(\LLV \fibprod{M} \algcurv^{\oplus r}\) with the same codimension as the codimension of the zero section in \((S_2E)^{\oplus r}\).
This zero section has codimension \(r(n-1)(n-2)/2\) since \(E\) has rank \(n-2\).
The set \(\nongenz{r}\) coincides by definition with this inverse image, completing the proof.
\end{proof}

\section{Transversality}\label{sec:transversality}
When \(f \colon M \to N\) is a smooth map and \(W \subseteq N\) is a smooth submanifold we use the notation \(f \pitchfork W\) to mean that \(f\) is transverse to \(W\).
This means that if \(p \in M\) is such that \(f(p) \in W\), then it holds that \(T_{f(x)}N = T_{f(x)}W + f_*(T_pM)\).
For our purposes, the most important consequence of this definition is that if \(f \pitchfork W\) and \(\dim(M) < \codim(W)\), then \(\im(f) \cap W = \emptyset\).
For details on transversality, see \cite[Chapter II]{GolubitskyGuillemin} and \cite[Chapter 3]{Hirsch}.

In this section, we will use the following transversality theorem to prove \maintheoremref.
It is a version of the Thom Transversality Theorem, which can be proved in the same way as \cite[Chapter II, Theorem 4.9]{GolubitskyGuillemin}, or by using \cite[Theorem 2.3.2]{EliashbergMishachev}.
\begin{prop}
Let \(E \to M\) and \(E' \to M\) be smooth fiber bundles.
Suppose that \(W\) is a submanifold of the fiber product \(E \fibprod{M} J^mE'\).
Let \(k \in \Z \cup \{\infty\}\) be such that \(k \geq m + 1\).
Endow \(\Gamma^\infty(E')\) with the Whitney \(C^k\)-topology and let
\[
G
=
\{\phi \in \Gamma^\infty(E') \mid (\id_E, j^m\phi \circ \pi) \pitchfork W\},
\]
where \(\pi \colon E \to M\) denotes the projection.
Then \(G\) is a residual subset of \(\Gamma^\infty(E')\).
\end{prop}
By letting \(m = 1 + r\), \(E = \oTM\) and \(E' = L\) we obtain the following.
\begin{lem}\label{lem:transversality}
Let \(M\) be a smooth manifold and let \(\pi \colon \oTM \to M\) be the projection.
Let \(W \subseteq \oTM \fibprod{M} J^{1+r}L\) be a smooth submanifold.
Let \(G\) be the set of Lorentzian metrics \(g \in \Gamma^\infty(L)\) such that the map \(\rho_g \colon \oTM \to \oTM \fibprod{M} J^{1+r}L\) defined by \(\rho_g(X) = (X, j^{1+r}_{\pi(X)}g)\) is transverse to \(W\).
Let \(k \in \Z \cup \{\infty\}\) be such that \(k \geq r + 2\).
Then \(G\) is residual in the Whitney \(C^k\)-topology on the space \(\Gamma^\infty(L)\) of all Lorentzian metrics.
\end{lem}

We can now prove that generic metrics have no \(r\)-nongeneric vectors if \(r\) is sufficiently large.
\mainthm*
\begin{proof}
\newcommand{\Wneq}{W_{\neq}}
\newcommand{\Wz}{W_{=}}
Let
\[\Wneq = \alpha_r^{-1}(\nongenneq{r}),\]
\[\Wz = \alpha_r^{-1}(\nongenz{r}).\]
Since \(\alpha_r\) is a submersion by \autoref{lem:alpha-submersion}, these sets are submanifolds with the same codimensions as \(\nongenneq{r}\) and \(\nongenz{r}\):
\[\codim(\Wneq) = \frac{rn(n-1)}{2},\]
\[\codim(\Wz) = \frac{r(n-1)(n-2)}{2} + 1.\]
By \autoref{lem:transversality} there is a set \(G \subseteq \Gamma^\infty(L)\), residual in the Whitney \(C^k\)-topology, of Lorentzian metrics \(g\) such that the map \({\rho_g \colon \oTM \to \oTM \fibprod{M} J^{1+r}L}\) defined by \(\rho_g(X) = (X, j^{1+r}_{\pi(X)}g)\) is transverse to \(\Wneq\) and \(\Wz\).
The manifold \(\oTM\) has dimension \(2n\).
We have assumed that \(n \geq 3\) and that
\[
r > \frac{4n - 2}{(n-1)(n-2)}
\]
so it holds that \(2n < \codim(\Wneq)\) and \(2n < \codim(\Wz)\).
Hence the transversality means that \(\im \rho_g\) is actually disjoint from both \(\Wneq\) and \(\Wz\).
Since \(\Wneq \cup \Wz = \alpha_r^{-1}(\nongen{r})\) we have now proved that \(\alpha_r(\im \rho_g) \cap \nongen{r} = \emptyset\), in other words that the metric \(g\) has no \(r\)-nongeneric vectors.
Since \(G\) is residual in the Whitney \(C^\infty\)-topology, it is dense in the Whitney \(C^\infty\)-topology (by \autoref{prop:baire}) and hence dense in the Whitney \(C^k\)-topology.
This completes the proof of the first part of the theorem.

Let \(g \in G\) and let \(\gamma\) be a \(g\)-geodesic.
That every non-zero vector is \(r\)-nongeneric implies that for each \(t \in \dom(\gamma)\) there is some \(k\) such that
\[
\nabla^k_{\dot\gamma(t)} \left((\doubleflat{\dot\gamma}) \cwedge R(\tensorslot, \dot\gamma, \tensorslot, \dot\gamma)\right) \neq 0.
\]
This means that the zeros of \(\left(\doubleflat{\dot\gamma}\right) \cwedge R(\tensorslot, \dot\gamma, \tensorslot, \dot\gamma)\) along \(\gamma\) form a discrete set.
In particular, each inextendible geodesic has at least one tangent vector which is generic, proving that the lightlike, timelike and spacelike generic conditions hold for \(g\).
\end{proof}
The inequality relating \(r\) and \(n\) may be weakened slightly if one instead wants to prove that the space of \(r\)-nongeneric vectors for generic metrics is \(0\)-dimensional.
This would also be sufficient to conclude that the generic condition holds for generic metrics.

The inequality can also be weakened if one is only interested in the timelike and spacelike generic conditions.

\appendix
\section{Jet bundles and Whitney topologies}\label{app:jets-and-whitney-topologies}
\subsection{Jet bundles}
As references for jets we suggest \cite[Chapter 1]{EliashbergMishachev}, \cite[Chapter II, \S2]{GolubitskyGuillemin} and \cite{Saunders}.

\begin{defn}
Let \(E \to B\) be a smooth fiber bundle. A \emph{local section} around a point \(b \in B\) is an open neighborhood \(U\) of \(b\) together with a section \(s \colon U \to \left.E\right|_U\) of the restricted fiber bundle \(\left.E\right|_U \to U\).
\end{defn}

A construction of \(k\)-jet bundles, complete with their fiber bundle structures, can be found in \cite[Chapter 6]{Saunders}.
As sets they are described as follows.
\begin{defn}
Let \(\pi \colon E \to B\) be a smooth fiber bundle, and let \(k\) be a nonnegative integer. 
The \emph{\(k\)-jet bundle} \(J^kE\) is a fiber bundle over \(B\) whose fiber \(J^k_bE\) over \(b \in B\) is the set of equivalence classes of local sections around \(b\) under the relation defined by letting \(\sigma_1 \sim \sigma_2\) if
\begin{itemize}
\item \(\sigma_1(b) = \sigma_2(b)\),
\item after choosing suitable coordinates (using a local trivialization) for \(B\) and \(E\) around \(b\) and \(\sigma_1(b) = \sigma_2(b)\), and locally viewing the sections \(\sigma_i\) as functions \(f_i\) from the trivializing neighborhoods to a fiber, the partial derivatives of \(f_1\) and \(f_2\) at \(b\) agree up to order \(k\).
\end{itemize}
We use the notation \(J^kE\) instead of the more proper \(J^k\pi\) since all jets in this paper are with respect to the same base space \(M\).
Note that \(J^0E\) can be identified with \(E\) itself.
\end{defn}

There is a natural fiber bundle map
\begin{center}
\begin{tikzpicture}[scale=1.4]
\node (Jk1E) at (0,1) {\(J^{k+1}E\)};
\node (JkE) at (2,1) {\(J^kE\)};
\node (B) at (1,0) {\(B\)};
\path[->,>=angle 90]
(Jk1E) edge node[above]{} (JkE)
(Jk1E) edge node[above]{} (B)
(JkE) edge node[above]{} (B);
\end{tikzpicture}
\end{center}
defined by discarding information about the highest-order partial derivatives. This map also allows us to view \(J^{k+1}E\) as a fiber bundle over \(J^kE\).

A section \(s\) of a fiber bundle \(E \to B\) gives rise to a section \(j^ks\) of the fiber bundle \(J^kE \to B\), for each nonnegative integer \(k\).
The section \(j^ks\) is called the \emph{prolongation} of \(s\).
Evaluating \(j^ks\) at a point \(b \in B\) gives an element \(j^k_bs \in J^k_bE\) which we call the \emph{\(k\)-jet of \(s\) at \(b\)}.

Analogously to the definition above, one can define the space \(J^k(M, N)\) of jets of arbitrary smooth maps between smooth manifolds \(M\) and \(N\).
See \cite[Chapter II, \S2]{GolubitskyGuillemin} for details.

\subsection{Whitney topologies}
We will define a family of topologies called the \emph{Whitney \(C^k\)-topologies} on function spaces \(C^\infty(M, N)\).
For details, see \cite[Chapter II, \S3]{GolubitskyGuillemin}. 
\begin{defn}
Let \(M\) and \(N\) be smooth manifolds.
Fix some nonnegative integer \(k\).
For each open subset \(U\) of \(J^k(M, N)\), let
\[
M^k(U)
=
\left\{ f \in C^\infty(M, N) \mid j^kf(M) \subseteq U \right\}.
\]

The collection
\[
\left\{ M^k(U) \mid U \subseteq J^k(M, N) \text{ is open} \right\}
\]
forms a basis for the \emph{Whitney \(C^k\)-topology} on \(C^\infty(M, N)\).

The collection
\[
\left\{ M^k(U) \mid k \geq 0 \text{ and } U \subseteq J^k(M, N) \text{ is open} \right\}\]
forms a basis for the \emph{Whitney \(C^\infty\)-topology} on \(C^\infty(M, N)\).
\end{defn}
\begin{defn}
A subset of a topological space is \emph{residual} if it is a countable intersection of open dense sets.
A topological space is a \emph{Baire space} if every residual set is dense.
\end{defn}
\begin{prop}[\protect{\cite[Chapter II, Proposition 3.3]{GolubitskyGuillemin}}]\label{prop:baire}
Let \(M\) and \(N\) be smooth manifolds.
Then \(C^\infty(M, N)\) endowed with the Whitney \(C^\infty\)-topology is a Baire space.
\end{prop}

\subsection*{Acknowledgements}
I wish to thank Marc Nardmann for many helpful comments, and in particular for suggesting a more elegant way of proving Propositions~\ref{prop:manifoldness-nonlightlike} and \ref{prop:manifoldness-lightlike}.

\bibliographystyle{abbrv}
\bibliography{References}

\begin{thebibliography}{10}

\bibitem{BeemHarris93a}
J.~K. Beem and S.~G. Harris.
\newblock The generic condition is generic.
\newblock {\em Gen. Relativity Gravitation}, 25(9):939--962, 1993.

\bibitem{BeemHarris93b}
J.~K. Beem and S.~G. Harris.
\newblock Nongeneric null vectors.
\newblock {\em Gen. Relativity Gravitation}, 25(9):963--973, 1993.

\bibitem{Besse}
A.~L. Besse.
\newblock {\em Einstein manifolds}.
\newblock Classics in Mathematics. Springer-Verlag, Berlin, 2008.
\newblock Reprint of the 1987 edition.

\bibitem{EliashbergMishachev}
Y.~Eliashberg and N.~Mishachev.
\newblock {\em Introduction to the {$h$}-principle}, volume~48 of {\em Graduate
  Studies in Mathematics}.
\newblock American Mathematical Society, Providence, RI, 2002.

\bibitem{GolubitskyGuillemin}
M.~Golubitsky and V.~Guillemin.
\newblock {\em Stable mappings and their singularities}, volume~14 of {\em
  Graduate Texts in Mathematics}.
\newblock Springer-Verlag, New York-Heidelberg, 1973.

\bibitem{HawkingEllis}
S.~W. Hawking and G.~F.~R. Ellis.
\newblock {\em The large scale structure of space-time}, volume~1 of {\em
  Cambridge Monographs on Mathematical Physics}.
\newblock Cambridge University Press, London, 1973.

\bibitem{Hirsch}
M.~W. Hirsch.
\newblock {\em Differential topology}, volume~33 of {\em Graduate Texts in
  Mathematics}.
\newblock Springer-Verlag, New York, 1994.
\newblock Corrected reprint of the 1976 original.

\bibitem{Rendall88a}
A.~D. Rendall.
\newblock The continuous determination of spacetime geometry by the {R}iemann
  curvature tensor.
\newblock {\em Classical Quantum Gravity}, 5(5):695--705, 1988.

\bibitem{Rendall88b}
A.~D. Rendall.
\newblock Curvature of generic space-times in general relativity.
\newblock {\em J. Math. Phys.}, 29(7):1569--1574, 1988.

\bibitem{Saunders}
D.~J. Saunders.
\newblock {\em The geometry of jet bundles}, volume 142 of {\em London
  Mathematical Society Lecture Note Series}.
\newblock Cambridge University Press, Cambridge, 1989.

\bibitem{Senovilla98}
J.~M.~M. Senovilla.
\newblock Singularity theorems and their consequences.
\newblock {\em Gen. Relativity Gravitation}, 30(5):701--848, 1998.

\end{thebibliography}

\end{document}